\documentclass[12pt]{article}
\usepackage{amsmath,amssymb,amsfonts,amsthm,pgf,tikz,mathrsfs}
\usepackage{mathdots}
\usepackage[affil-it]{authblk}

\newtheorem{theorem}{Theorem}[section]
\newtheorem{lemma}[theorem]{Lemma}
\newtheorem{corollary}[theorem]{Corollary}
\newtheorem{proposition}[theorem]{Proposition}

\theoremstyle{definition}

\newtheorem{example}[theorem]{Example}
\newtheorem{obs}[theorem]{Observation}
\newtheorem{question}[theorem]{Question}

\theoremstyle{remark}

\numberwithin{equation}{section}

\newcommand{\disp}{\displaystyle}

\definecolor{dkgreen}{rgb}{0,0.5,0}
\definecolor{ltgreen}{rgb}{.6,1,0.6}
\definecolor{dkred}{rgb}{0.6,0.0,0}
\definecolor{dkblue}{rgb}{0,0.0,0.6}
\definecolor{ltblue}{rgb}{0.6,0.6,1}
\definecolor{gray}{rgb}{0.5,0.5,0.5}
\definecolor{mauve}{rgb}{0.58,0,0.82}
\definecolor{magenta}{rgb}{1,0,.5}

\usepackage{setspace}
\usepackage{color,soul}

\begin{document}
\title{An Analog of Matrix Tree Theorem for Signless Laplacians}
\author[1]{Keivan Hassani Monfared \thanks{The work of this author was partially supported by the Natural Sciences and Engineering Research Council of Canada (NSERC) and a postdoctoral fellowship of the Pacific Institute of Mathematical  Sciences (PIMS).}}
\affil[1]{\small Department of Mathematics and Statistics, University of Calgary, 2500 University Drive NW, Calgary, AB, T2N 1N4, Canada\\
k1monfared@gmail.com}

\author[2]{Sudipta Mallik}
\affil[2]{ Department of Mathematics and Statistics, Northern Arizona University, 801 S. Osborne Dr. PO Box: 5717, Flagstaff, AZ 86011, USA\\
sudipta.mallik@nau.edu}

\maketitle

\begin{abstract}
A spanning tree of a graph is a connected subgraph on all vertices with the minimum number of edges. The number of spanning trees in a graph $G$ is given by Matrix Tree Theorem in terms of principal minors of Laplacian matrix of $G$. We show a similar combinatorial interpretation for principal minors of signless Laplacian $Q$. We also prove that the number of odd cycles in $G$ is less than or equal to $\frac{\det(Q)}{4}$, where the equality holds if and only if $G$ is a bipartite graph or an odd-unicyclic graph.
\end{abstract}

\renewcommand{\thefootnote}{\fnsymbol{footnote}} 
\footnotetext{\emph{2010 Mathematics Subject Classification. 05C50,65F18}}    
\renewcommand{\thefootnote}{\arabic{footnote}}

\section{Introduction}
	For a simple graph $G$ on $n$ vertices $1,2,\ldots,n$ and $m$ edges $1,2,\ldots, m$ we define its \emph{degree matrix} $D$,  \emph{adjacency matrix} $A$, and \emph{incidence matrix} $N$ as follows:
	\begin{enumerate}
		\item $D=[d_{ij}]$ is an $n \times n$ diagonal matrix where $d_{ii}$ is the degree of the vertex $i$ in $G$ for $i=1,2,\ldots,n$.
	
		\item $A=[a_{ij}]$ is an $n \times n$ matrix with zero diagonals where $a_{ij}=1$ if vertices $i$ and $j$ are adjacent in $G$ and $a_{ij}=0$ otherwise for $i,j=1,2,\ldots,n$.
		
		\item $N=[n_{ij}]$ is an $n \times m$ matrix whose rows are indexed by vertices and columns are indexed by edges of $G$. The entry $n_{ij} = 1$ whenever vertex $i$ is incident with edge $j$ (i.e., vertex $i$ is an endpoint of edge $j$) and $n_{ij}=0$ otherwise.
	\end{enumerate}
	
	We define the {\it Laplacian matrix} $L$ and {\it signless Laplacian matrix} $Q$ to be $L=D-A$ and $Q=D+A$, respectively. It is well-known that both $L$ and $Q$ have nonnegative real eigenvalues \cite[Sec 1.3]{BH}.  	Note the relation between the spectra of $L$ and $Q$:
	\begin{theorem}\cite[Prop $1.3.10$]{BH}\label{CB2}
	Let $G$ be a simple graph on $n$ vertices. Let $L$ and $Q$ be the Laplacian matrix and the signless Laplacian matrix of $G$, respectively, with eigenvalues $0=\mu_1\leq \mu_2\leq \cdots \leq \mu_n$ for $L$, and $\lambda_1\leq \lambda_2\leq \cdots \leq \lambda_n$ for $Q$. Then $G$ is bipartite if and only if 
	$\{\mu_1,\mu_2, \ldots ,\mu_n\}=\{\lambda_1, \lambda_2, \ldots ,\lambda_n\}.$
	\end{theorem}

	\begin{theorem}\cite[Prop $2.1$]{CRC}
	The smallest eigenvalue of the signless Laplacian of a connected graph is equal to $0$ if and only if the graph is bipartite. In this case $0$ is a simple eigenvalue.
	\end{theorem}
	
We use the following notation for submatrices of an $n \times m$ matrix $M$: for sets $I \subset \{ 1,2,\ldots,n\}$ and $J\subset \{ 1,2,\ldots,m\}$, 
	\begin{itemize}
		\item $M[I ; J]$ denotes the submatrix of $M$ whose rows are indexed by $I$ and columns are indexed by $J$.
		\item $M(I ; J)$ denotes the submatrix of $M$ obtained by removing the rows indexed by $I$ and removing the columns indexed by $J$.
		\item $M(I ; J]$ denotes the submatrix of $M$ whose columns are indexed by $J$, and obtained by removing rows indexed by $I$.
	\end{itemize}
	
	We often list the elements of $I$ and $J$, separated by commas in this submatrix notation, rather than writing them as sets. For example, $M(2 ; 3,7,8]$ is a $(n-1) \times 3$ matrix whose rows are the same as the rows of $M$ with the  the second row deleted and columns are respectively the third, seventh, and eighth columns of $M$. Moreover, if $I=J$, we abbreviate $M(I ; J)$ and $M[I ; J]$ as $M(I)$ and $M[I]$ respectively. Also we abbreviate $M(\varnothing ; J]$ and $M(I;\varnothing)$ as $M(;J]$ and $M(I;)$ respectively. 
	
	A {\it spanning tree} of $G$ is a connected subgraph of $G$ on all $n$ vertices with minimum number of edges which is $n-1$ edges. The number of spanning trees in a graph $G$ is denoted by $t(G)$ and is given by  Matrix Tree Theorem:
	
	\begin{theorem}[Matrix Tree Theorem]\cite[Prop $1.3.4$]{BH}\label{MTT}
	 Let $G$ be a a simple graph on $n$ vertices and $L$ be the Laplacian matrix of $G$ with eigenvalues $0=\mu_1\leq \mu_2\leq \cdots \leq \mu_n$. Then the number $t(G)$ of spanning trees of $G$ is 
	 $$t(G)=\det\left(L(i)\right)=\frac{\mu_2\cdot \mu_3 \cdots \mu_n}{n},$$
	 for all $i=1,2,\ldots,n$.
	\end{theorem}
	
	We explore if there is an analog of the Matrix Tree Theorem for the signless Laplacian matrix $Q$. First note that unlike $\det\left(L(i)\right)$,  $\det\left(Q(i)\right)$ is not necessarily the same for all $i$ as illustrated in the following example.
	
	\begin{example}
	For the paw graph $G$ with its signless Laplacian matrix $Q$ in Figure \ref{paw},  $\det\left(Q(1)\right)=7\neq 3=\det\left(Q(2)\right)=\det\left(Q(3)\right)=\det\left(Q(4)\right)$.
	\begin{figure}
	\begin{center}
	\begin{tikzpicture}[colorstyle/.style={circle, fill, black, scale = .5}]
		
		\node (1) at (0,1)[colorstyle, label=above:$1$]{};
		\node (2) at (0,0)[colorstyle, label=below:$2$]{};
		\node (3) at (-1,-1)[colorstyle, label=below:$3$]{};
		\node (4) at (1,-1)[colorstyle, label=below:$4$]{};	
		
		\draw [] (1)--(2)--(3)--(4)--(2);
		\node at (2,0)[label=right:{$Q=\left[\begin{array}{cccc}
	1&1&0&0\\
	1&3&1&1\\
	0&1&2&1\\
	0&1&1&2
	\end{array}\right] $}]{};
	\end{tikzpicture}
	\caption{Paw $G$ and its signless Laplacian matrix $Q$}\label{paw}
	\end{center}
	\end{figure}
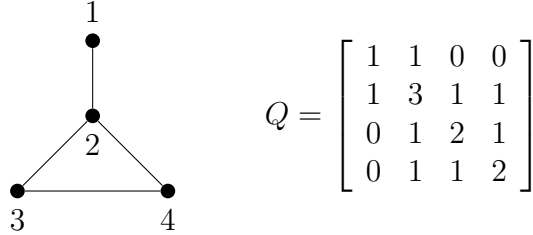
	\end{example}

	The Matrix Tree Theorem can be proved by the Cauchy-Binet formula:
	\begin{theorem}[Cauchy-Binet]\cite[Prop $1.3.5$]{BH}\label{CB}
		Let $m \leq n$. For $m\times n$ matrices $A$ and $B$, we have
		\[
			\det(AB^T)=\sum_{S} \det(A(;S]) \det(B(;S]),
		\]
		where the summation runs over $\binom{n}{m}$ $m$-subsets $S$ of $\{1,2,\ldots,n\}$.
	\end{theorem}

The following observation provides a decomposition of the signless Laplacian matrix $Q$ which enables us to apply the Cauchy-Binet formula on it.
	
\begin{obs}\label{CB3}
		Let $G$ be a simple graph on $n\geq 2$ vertices with $m$ edges, and $m \geq n-1$. Suppose $N$ and $Q$ are the incidence matrix  and  signless Laplacian matrix of $G$, respectively. Then
		\begin{enumerate}
			\item[(a)] $Q=NN^T$,
			\item[(b)] $Q(i)=N(i;)N(i;)^T$, $i=1,2,\ldots,n$, and 
			\item[(c)] $\det(Q(i))=\det(N(i;)N(i;)^T)=\sum_{S} \det(N(i;S])^2,$
			where the summation runs over all $(n-1)$-subsets $S$ of $\{1,2,\ldots,m\}$, (by Cauchy-Binet formula \ref{CB}).
		\end{enumerate}
\end{obs}

\section{Principal minors of signless Laplacians}
	In this section we find a combinatorial formula for a principal minor $\det(Q(i))$ for the signless Laplacian matrix $Q$ of a given graph $G$. We mainly use  Observation \ref{CB3}(c) given by Cauchy-Binet formula which involves determinant of submatrices of incidence matrices. This approach is completely different from the methods applied for related spectral results in \cite{CRC}. But we borrow the definition of $TU$-subgraphs from \cite{CRC} slightly modified as follows:  A {\it $TU$-graph} is a graph whose connected components are trees or odd-unicyclic graphs. A {\it $TU$-subgraph} of $G$ is a spanning subgraph of $G$ that is a $TU$-graph. The following lemma finds the number of trees in a $TU$-graph.
	
	\begin{lemma}\label{noofcyclesinTU}
		If $G$ is a $TU$-graph on $n$ vertices with $n-k$ edges consisting of $c$ odd-unicyclic graphs and $s$ trees,  then $s=k$.
	\end{lemma}
	\begin{proof}
		Suppose the number vertices of the cycles are $n_1,n_2,\ldots,n_c$ and that of the trees are $t_1,t_2,\ldots,t_s$. Then the total number of edges is 
		\begin{align*}
			n-k = \sum_{i=1}^{c} n_i + \sum_{i=1}^{s} (t_i-1) = n - s
		\end{align*}
		which implies $s = k$.
	\end{proof}
	
Now we find the determinant of incidence matrices of some special graphs in the following lemmas.
	\begin{lemma}\label{incidencedet}
		If $G$ is an odd (resp. even) cycle, then the determinant of its incidence matrix is $\pm 2$ (resp. zero). 
	\end{lemma}
	\begin{proof}
	Let $G$ be a cycle with the incidence matrix $N$. Then up to permutation we have 
	$$N=PN'Q=P
	\left[\begin{array}{cccccc}
	1&0&0&\cdots&0&1\\
	1&1&0&\cdots&0&0\\
	0&1&1&\ddots&\vdots&\vdots\\
	\vdots&\ddots &\ddots&\ddots&\ddots&\vdots\\
	\vdots&\vdots&\ddots&\ddots&1&0\\
	0&\cdots&\cdots&0&1&1\\
	\end{array}\right]Q,
	$$
	for some permutation matrices $P$ and $Q$. By a cofactor expansion across the first row we have  
	$$\det(N)=\det(P)\det(N')\det(Q)=(\pm 1)(1+(-1)^{n+1})(\pm 1).$$
	If $n$ is odd (resp. even), then $\det(N)= \pm 2$ (resp. zero).
	\end{proof}
	
	\begin{lemma}\label{incidencedetoddcyclic}
	If $G$ is an odd unicyclic (resp. even unicyclic) graph, then the determinant of its incidence matrix is $\pm 2$ (resp. $0$). 
	\end{lemma}
	\begin{proof}
	Let $G$ be a unicyclic graph with incidence matrix $N$ and $t$ vertices not on the cycle. We prove the statement by induction on $t$. If $t = 0$, then $G$ is an odd (resp. even) cycle and then $\det(N_i) = \pm 2$ (resp. $0$) by Lemma \ref{incidencedet}. Assume the statement holds for some $t \geq 0$. Let $G$ be a unicyclic graph with $t+1$ vertices not on the cylce. Then $G$ has a pendant vertex, say vertex $i$. The vertex $i$ is incident with exactly one edge of $G$, say $e_l = \{ i,j \}$. Then $i$th row of $N$ has only one nonzero entry which is the $(i,l)$th entry and it is equal to $1$. To find $\det(N)$ we have a cofactor expansion across the $i$th row and get 
		\[
			\det(N)=\pm 1 \cdot \big( \pm \det(N(i ; l)) \big).
		\]
		Note that $N(i ; l)$ is the incident matrix of $G(i)$, which is a unicyclic graph with $t$ vertices not on the cycle. By induction hypothesis, $\det(N(i ; l)) = \pm 2$ (resp. $0$).  Thus $\det(N)=\pm 1 \cdot \big( \pm \det(N(i ; l)) \big)=\pm 2$ (resp. $0$). 
	\end{proof}

By a similar induction on the number of pendant vertices we get the following result.
\begin{lemma}\label{tree_det}
Let $H$ be a tree with at least one edge and $N$ be the incidence matrix of $H$. Then $\det(N(i;)) = \pm 1$ for all vertices $i$ of $H$.
\end{lemma}

\begin{lemma}\label{lem_tree_and_one_more_edge}
		Let $H$ be a graph on $n$ vertices and $n-1$ edges with incidence matrix $N$. If $H$ has a connected component which is a tree and an edge which is not on the tree, then $\det(N(i;)) = 0$ for all vertices $i$ not on the tree.
\end{lemma}
\begin{proof} 		
Let $H$ have a connected component $T$ which is a tree and an edge $e_j$ which is not on $T$. Suppose $i$ is a vertex of $G$ that is not on $T$. If $T$ consists of just one vertex, then the corresponding row in $N(i;)$ is a zero row giving $\det(N(i;))=0$. Suppose  $T$ has at least two vertices. 
		Now consider  the square submatrix $N'$ of $N(i;)$ with rows corresponding to verteices of $T$ and columns corresponding to edges of $T$ together with $e_j$. Then the column of $N'$ corresponding to $e_j$ is a zero row giving $\det(N')=0$. Since entries in rows of $N_i[S]$ corresponding to $T$ that are outside of $N'$ are zero, the rows of $N(i;)$ corresponding to $T$ are linearly dependent and consequently $\det(N(i;))=0$.
\end{proof}

	Now we break down different scenarios that can happen to a graph with $n$ vertices and $m=n-1$ edges.
	
	\begin{proposition}\label{prop_enn_vert_ennminusone_edg}
		Let $H$ be a graph on $n$ vertices and $m = n-1$ edges. Then one of the following is true for $H$.
		\begin{enumerate}
			\item \label{prop_enn_vert_ennminusone_edg_case_tree} $H$ is a tree.
			\item \label{prop_enn_vert_ennminusone_edg_case_even_cycle} $H$ has an even cycle and a vertex not on the cycle.
			\item \label{prop_enn_vert_ennminusone_edg_case_non_tu} $H$ has no even cycles, but $H$ has a connected component with at least two odd cycles and at least two connected components which are trees.
			\item \label{prop_enn_vert_ennminusone_edg_case_tu} $H$ is a disjoint union of odd unicyclic graphs and exactly one tree, i.e., $H$ is a $TU$-graph.
		\end{enumerate}
	\end{proposition}
	
	\begin{proof}
		If $H$ is connected then it is a tree. This implies Case \ref{prop_enn_vert_ennminusone_edg_case_tree}. Now assume $H$ is not connected. If $H$ has no cycles, then it is a forest with at least two connected components. This would imply that $m < n-2$, contradicting the assumption that $m = n-1$. Thus $H$ has at least one cycle. Suppose $H$ has $t \geq 2$ connected components $H_i$ with $m_i$ edges and $n_i$ vertices, where the first $k$ of them have at least a cycle and the rest are trees. For $i= 1,\ldots, k$, $H_i$  has $m_i \geq n_i$. Note that
\begin{equation}\label{treenumber}
-1 = m - n = \sum_{i=1}^{t} (m_i - n_i) =  \sum_{i=1}^{k} (m_i - n_i) + \sum_{i=k+1}^{t} (m_i - n_i)
\end{equation}		
		
Since $H_i$  has a cycle for $i =1,\ldots,k$ and $H_i$ is a tree for $i = k+1,\ldots,t$,
		\[
			\ell := \sum_{i=1}^{k} (m_i - n_i) \geq 0,
		\]
		and 
		\[ 
			\sum_{i=k+1}^{t} (m_i - n_i) = -(t - k).
		\] 
Then $t-k = \ell + 1$ by (\ref{treenumber}). In other words, in order to make up for the extra edges in the connected components with cycles, $H$ has to have exactly $\ell + 1$ connected components which are trees. 
		
		 If $H$ has an even cycle, then $\ell \geq 0$ and hence $t-k \geq 1$. This means there is at least one connected component which is tree and it contains a vertex which is not in the cycle. This implies Case \ref{prop_enn_vert_ennminusone_edg_case_even_cycle}. Otherwise, all of the cycles of $H$ are odd. If it has more than one cycle in a connected component, then $\ell \geq 1$ and thus $t-k \geq 2$. This implies Case \ref{prop_enn_vert_ennminusone_edg_case_non_tu}. Otherwise, each $H_i$ with $i=1,\ldots,k$ has exactly one cycle in it, which implies $\ell = 0$, and then $t-k =1$. This implies Case \ref{prop_enn_vert_ennminusone_edg_case_tu}.
	\end{proof}

	\begin{theorem}\label{thm_casesQ_redo} 
		Let $G$ be a simple connected graph on $n\geq 2$ vertices and $m$ edges with the incidence matrix $N$. Let $i$ be an integer from $\{1,2,\ldots,n\}$. Let  $S$  be an $(n-1)$-subset of $\{1,2,\ldots,m\}$ and $H$ be a spanning subgraph of $G$ with edges indexed by $S$. Then one of the following holds for $H$.
		\begin{enumerate}
			\item \label{thm_casesQ_redo_case_tree}  $H$ is a tree. Then $\det(N(i;S]) = \pm 1$.
			\item \label{thm_casesQ_redo_case_even_cycle} $H$ has an even cycle and a vertex not on the cycle. Then $\det(N(i;S]) = 0$.
			\item \label{thm_casesQ_redo_case_non_tu} $H$ has no even cycles, but it has a connected component with at least two odd cycles and at least two connected components which are trees. Then $\det(N(i;S]) = 0$.
			\item \label{thm_casesQ_redo_case_tu} $H$ is a $TU$-subgraph of $G$ consisting of $c$  odd-unicyclic graphs $U_1, U_2, \ldots , U_c$ and a unique tree $T$. If $i$ is a vertex of $U_j$ for some $j=1,2,\ldots,c$, then $\det(N(i;S]) =0$. If $i$ is a vertex of $T$, then  $\det(N(i;S]) =\pm 2^c$.	
		\end{enumerate}
	\end{theorem}
	
\begin{proof}
	Suppose vertices and edges of $G$ are $1,2,\ldots,n$ and $e_1,e_2,\ldots,e_m$, respectively. Note that $m\geq n-1$ since $G$ is connected.  
	\begin{enumerate}
		\item  Suppose $H$ is a tree. Since $n\geq 2$, $H$ has an edge. Then by Lemma \ref{tree_det},  $\det(N(i;S]) = \pm 1$.

		\item Suppose $H$ contains an even cycle $C$ as a subgraph and a vertex $j$ not on $C$. 
		
		{\it Case 1.} Vertex $i$ is not in $C$\\
		Then the square submatrix $N'$ of $N(i;S]$ corresponding to $C$ has determinant zero by Lemma \ref{incidencedetoddcyclic}. Since entries in columns of $N(i;S]$ corresponding to $C$ that are outside of $N'$ are zero, the columns of $N(i;S]$ corresponding to $C$ are linearly dependent and consequently $\det(N(i;S])=0$. \\
		
		{\it Case 2.} Vertex $i$ is in $C$\\
		Since $i$ is in $C$, we have $j\neq i$. Consider the square submatrix $N'$ of $N(i;S]$ that has rows corresponding to vertex $j$ and vertices of $C$ excluding $i$  and columns corresponding to edges of $C$. Since vertex $j$ is not on $C$, the row of $N'$ corresponding to vertex $j$ is a zero row and consequently $\det(N')=0$. Since entries in columns of $N_i[S]$ corresponding to $C$ that are outside of $N'$ are zero, the columns of $N(i;S]$ corresponding to $C$ are linearly dependent and consequently $\det(N(i;S])=0$.

	\item Suppose $H$ has no even cycles, but it has a connected component with at least two odd cycles and at least two connected components which are trees.  Then vertex $i$ is not in one of the trees. Then $\det(N(i;S])=0$ by Lemma \ref{lem_tree_and_one_more_edge}.
	
		\item Suppose $H$ is a $TU$-subgraph of $G$ consisting of $c$  odd-unicyclic graphs $U_1, U_2, \ldots , U_c$ and a unique tree $T$. If $i$ is a vertex of $U_j$ for some $j = 1, \ldots, c$, then  $\det(N(i;S])=0$ by Lemma \ref{lem_tree_and_one_more_edge}. If $i$ is a vertex of the tree $T$, then $N(i;S]$ is a direct sum of incidence matrices of odd-unicyclic graphs $U_1, U_2, \ldots , U_c$  and the incidence matrix of the tree $T$ with one row deleted (which does not exist when $T$ is a tree on the single vertex $i$). By Lemma \ref{incidencedetoddcyclic} and \ref{tree_det}, $\det(N(i;S])=(\pm 2)^c\cdot (\pm 1)= \pm 2^c$.
	
	\end{enumerate}
	\end{proof}

	The preceding results are summarized in the following theorem.
	\begin{theorem}\label{mainlemma}
	Let $G$ be a simple connected graph on $n\geq 2$ vertices and $m$ edges with the incidence matrix $N$. Let $i$ be an integer from $\{1,2,\ldots,n\}$. Let  $S$  be an $(n-1)$-subset of $\{1,2,\ldots,m\}$ and $H$ be a spanning subgraph of $G$ with edges indexed by $S$. 
	\begin{enumerate}
	\item[(a)] If $H$ is not a $TU$-subgraph of $G$, then $\det(N(i;S])=0$.
	\item[(b)] Suppose $H$ is a $TU$-subgraph of $G$ consisting of $c$  odd-unicyclic graphs $U_1, U_2, \ldots , U_c$ and a unique tree $T$. If $i$ is a vertex of $U_j$ for some $j=1,2,\ldots,c$, then $\det(N(i;S]) =0$. If $i$ is a vertex of $T$, then  $\det(N(i;S]) =\pm 2^c$.	
	\end{enumerate}
	\end{theorem}

	For a $TU$-subgraph $H$ of $G$, the number of connected components that are odd-unicyclic graphs is denoted by $c(H)$. So a $TU$-subgraph $H$ on $n-1$ edges with $c(H)=0$ is a spanning tree of $G$.
	
	\begin{theorem}\label{mainresult}
	Let $G$ be a simple connected graph on $n\geq 2$ vertices $1,2,\ldots,n$ with the signless Laplacian matrix $Q$. Then 
	$$\det(Q(i))=\sum_{H}4^{c(H)},$$
	 where the summation runs over all $TU$-subgraphs $H$ of $G$ with $n-1$ edges consisting of a unique tree on vertex $i$ and  $c(H)$ odd-unicyclic graphs.
	\end{theorem}
	\begin{proof}
	By Observation \ref{CB3}, we have,
	$$\det(Q(i))=\sum_{S} \det(N(i;S])^2,$$
where the summation runs over all $(n-1)$-subsets $S$ of $\{1,2,\ldots,m\}$. By Theorem \ref{mainlemma}, we have,	
	$$\det(Q(i))=\sum_{S} \det(N(i;S])^2=\sum_{H}(\pm 2^{c(H)})^2=\sum_{H}4^{c(H)},$$	
	where the summation runs over all $TU$-subgraphs $H$ of $G$ with $n-1$ edges consisting of a unique tree on vertex $i$ and  $c(H)$ odd-unicyclic graphs.
	\end{proof}

	\begin{example}
	Consider the Paw $G$ and its signless Laplacian matrix $Q$ in Figure \ref{paw}. To determine $\det(Q(1))$, consider the $TU$-subgraphs of $G$ with $3$ edges consisting of a unique tree on vertex $1$: $H_1$, $H_2$, $H_3$, $H_4$ in Figure \ref{tu}. Note $c(H_1)=c(H_2)=c(H_3)=0$ and $c(H_4)=1$. Then by Theorem \ref{mainresult},
	$$\det(Q(1))=\sum_{H}4^{c(H)}=4^{c(H_1)}+4^{c(H_2)}+4^{c(H_3)}+4^{c(H_4)}=4^{0}+4^{0}+4^{0}+4^{1}=7.$$

	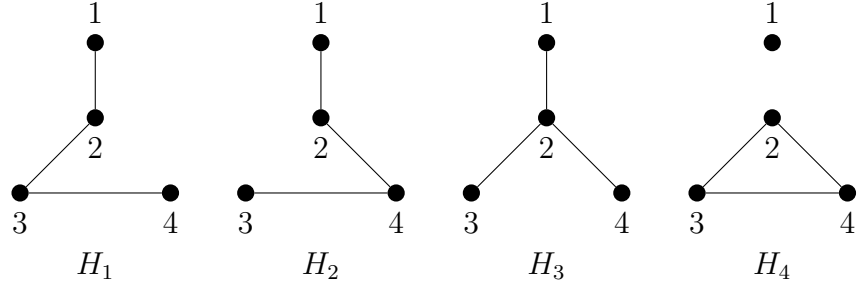
\begin{figure}
	\begin{center}
	\begin{tikzpicture}
	[scale=1,colorstyle/.style={circle, draw=black!100,fill=black!100, thick, inner sep=0pt, minimum size=2mm},>=stealth]
		\node (1) at (0,1)[colorstyle, label=above:$1$]{};
		\node (2) at (0,0)[colorstyle, label=below:$2$]{};
		\node (3) at (-1,-1)[colorstyle, label=below:$3$]{};
		\node (4) at (1,-1)[colorstyle, label=below:$4$]{};
		\node at (0,-1.5)[label=below:$H_1$]{};
		\draw [] (1)--(2)--(3)--(4);
		
		\node (1b) at (3,1)[colorstyle, label=above:$1$]{};
		\node (2b) at (3,0)[colorstyle, label=below:$2$]{};
		\node (3b) at (2,-1)[colorstyle, label=below:$3$]{};
		\node (4b) at (4,-1)[colorstyle, label=below:$4$]{};
		\node at (3,-1.5)[label=below:$H_2$]{};
		\draw [] (1b)--(2b)--(4b)--(3b);
		
		\node (1c) at (6,1)[colorstyle, label=above:$1$]{};
		\node (2c) at (6,0)[colorstyle, label=below:$2$]{};
		\node (3c) at (5,-1)[colorstyle, label=below:$3$]{};
		\node (4c) at (7,-1)[colorstyle, label=below:$4$]{};
		\node at (6,-1.5)[label=below:$H_3$]{};
		\draw [] (1c)--(2c)--(3c);	
		\draw [] (2c)--(4c);
		
		\node (1d) at (9,1)[colorstyle, label=above:$1$]{};
		\node (2d) at (9,0)[colorstyle, label=below:$2$]{};
		\node (3d) at (8,-1)[colorstyle, label=below:$3$]{};
		\node (4d) at (10,-1)[colorstyle, label=below:$4$]{};
		\node at (9,-1.5)[label=below:$H_4$]{};
		\draw [] (4d)--(2d)--(3d)--(4d);	
	\end{tikzpicture}
	\caption{$TU$-subgraphs of Paw $G$ with $3$ edges consisting of a unique tree on vertex $1$}\label{tu}
	\end{center}
	\end{figure}
	
	\end{example}

	\begin{corollary}
	Let $G$ be a simple connected graph on $n\geq 2$ vertices $1,2,\ldots,n$. Let $Q$ be the signless Laplacian matrix of $G$ with eigenvalues $\lambda_1\leq \lambda_2\leq \cdots \leq \lambda_n$. Then
	\begin{enumerate}
	\item[(a)]  $\det(Q(i))\geq t(G)$, the number of spanning trees of $G$, where the equality holds if and only if all odd cycles of $G$ contain vertex $i$.
	
	\item[(b)] $$\frac{1}{n}\disp\sum_{1\leq i_1<i_2<\cdots<i_n\leq n} \lambda_{i_1}\lambda_{i_2}\cdots \lambda_{i_{n-1}}
	=\frac{1}{n}\disp\sum_{i=1}^n\det(Q(i))\geq t(G),$$
	where the equality holds if and only if $G$ is an odd cycle or a bipartite graph.
	\end{enumerate}
	\end{corollary}
	
	\begin{proof}
	\begin{enumerate}
	\item[(a)] First note that a $TU$-subgraph $H$ on $n-1$ edges with $c(H)=0$ is a spanning tree of $G$. Then $\det(Q(i))=\sum_{H}4^{c(H)}\geq \sum_{T}4^{0}$, where the sum runs over all spanning trees $T$ of $G$ containing vertex $i$. So  $\det(Q(i))$ is greater than or equal to the number of spanning trees of $G$ containing vertex $i$. Since each spanning tree contains vertex $i$, $\det(Q(i))\geq t(G)$ where the equality holds if and only if all odd-unicyclic subgraphs of $G$ contain vertex $i$ by Theorem \ref{mainresult}. Finally note that all odd-unicyclic subgraphs of $G$ contain vertex $i$ if and only if all odd cycles of $G$ contain vertex $i$.
	
	\item[(b)]  The first equality follows from the well-known linear algebraic result 
		\[	
			\disp\sum_{1\leq i_1<i_2<\cdots<i_n\leq n} \lambda_{i_1}\lambda_{i_2}\cdots \lambda_{i_{n-1}}=\disp\sum_{i=1}^n\det(Q(i)).
		\]
		Now by (a) $\det(Q(i))\geq t(G)$ where the equality holds if and only if all odd cycles of $G$ contain vertex $i$. Then 
		\[
			\frac{1}{n}\disp\sum_{i=1}^n\det(Q(i))\geq t(G)
		\] 
		where the equality holds if and only if $\det(Q(i))=t(G)$ for all $i=1,2,\ldots,n$. So the equality holds if and only if all odd cycles of $G$ contain every vertex of  $G$ which means $G$ is an odd cycle or a bipartite graph ($G$ has no odd cycles).
	\end{enumerate}
	\end{proof}

\section{Number of odd cycles in a graph}
In this section we find a combinatorial formula for $\det(Q)$ for the signless Laplacian matrix $Q$ of a given graph $G$. As a corollary we show that the number of odd cycles in $G$ is less than or equal to $\frac{\det(Q)}{4}$. 

\begin{proposition}\label{prop_enn_vert_ennminusone_edg}
		Let $H$ be a graph on $n$ vertices and $m = n$ edges. Then one of the following is true for $H$.
		\begin{enumerate}
		\item $H$ has a connected component which is a tree.
		\item  All connected components of $H$ are unicyclic and at least one of them is even-unicyclic.
		\item All connected components of $H$ are odd-unicyclic.
		\end{enumerate}
\end{proposition}
	
\begin{proof}
Suppose $H$ has $t \geq 2$ connected components $H_i$ with $m_i$ edges and $n_i$ vertices, where the first $k$ of them have at least one cycle and the rest are trees. For $i= 1,\ldots, k$, $H_i$  has $m_i \geq n_i$. Note that
\begin{equation}\label{treenumber_m=n}
0= m - n = \sum_{i=1}^{t} (m_i - n_i) =  \sum_{i=1}^{k} (m_i - n_i) + \sum_{i=k+1}^{t} (m_i - n_i)
\end{equation}		
		
Since $H_i$  has a cycle for $i =1,\ldots,k$ and $H_i$ is a tree for $i = k+1,\ldots,t$,
		\[
			\ell := \sum_{i=1}^{k} (m_i - n_i) \geq 0,
		\]
		and 
		\[ 
			\sum_{i=k+1}^{t} (m_i - n_i) = -(t - k).
		\] 
Then $t-k = \ell$ by (\ref{treenumber_m=n}). If $H$ has a connected component which is a tree, we have Case 1. Otherwise $t-k =0$ which implies $\ell = \sum_{i=1}^{k} (m_i - n_i)=0$. Then $m_i= n_i$, for $i=1,2,\ldots,k$, i.e., all connected components of $H$ are unicyclic. If one of the unicyclic components is even-unicyclic, we get Case 2. Otherwise  all connected components of $H$ are odd-unicyclic which is Case 3. Finally if $H$ is connected, it is unicyclic and cosequently it is Case 2 or 3.
\end{proof}

\begin{lemma}\label{tree_and_one_more_edge}
Let $H$ be a graph on $n$ vertices and $n$ edges with incidence matrix $N$. If $H$ has a connected component which is a tree and an edge which is not on the tree, then $\det(N) = 0$.
\end{lemma}
\begin{proof} 		
Let $H$ have a connected component $T$ which is a tree and an edge $e_j$ which is not on $T$. If $T$ consists of just one vertex, say $i$, then row $i$ of $N$ is a zero row giving $\det(N)=0$. Suppose  $T$ has at least two vertices. 
		Now consider  the square submatrix $N'$ of $N$ with rows corresponding to vertices of $T$ and columns corresponding to edges of $T$ together with $e_j$. Then the column of $N'$ corresponding to $e_j$ is a zero row giving $\det(N')=0$. Since entries in rows of $N$ corresponding to $T$ that are outside of $N'$ are zero, the rows of $N$ corresponding to $T$ are linearly dependent and consequently $\det(N)=0$.
\end{proof}

\begin{theorem}\label{oddunicycle}
Let $G$ be a simple graph on $n$ vertices and $m\geq n$ edges with the incidence matrix $N$. Let  $S$  be a $n$-subset of $\{1,2,\ldots,m\}$ and $H$ be a spanning subgraph of $G$ with edges indexed by $S$. Then one of the following is true for $H$: 
\begin{enumerate}
		\item $H$ has a connected component which is a tree. Then $\det(N[S])=0$.
		\item  All connected components of $H$ are unicyclic and at least one of them is even-unicyclic. Then $\det(N[S])=0$.
		\item $H$ has $k$ connected components which are all odd-unicyclic. Then $\det(N[S])=\pm 2^k$.
		\end{enumerate}
\end{theorem}

\begin{proof}
$\;$
\begin{enumerate}
		\item Suppose $H$ has a connected component which is a tree. Since $H$ has $n$ edges, $H$ has an edge not on the tree. Then $\det(N[S])=0$ by Lemma \ref{tree_and_one_more_edge}.
		\item  Suppose all connected components of $H$ are unicyclic and at least one of them is even-unicyclic. Since $N[S]$ is a direct sum of incidence matrices of unicyclic graphs where at least one of them is even-unicyclic, then $\det(N[S])=0$ by Lemma \ref{incidencedet}.
		\item Suppose $H$ has $k$ connected components which are all odd-unicyclic. Since $N[S]$ is a direct sum of incidence matrices of $k$ odd-unicyclic graphs, then $\det(N[S])=(\pm 2)^k=\pm 2^k$ by Lemma \ref{incidencedet}.
		\end{enumerate}

\end{proof}

By Theorem \ref{CB} and \ref{oddunicycle}, we have the following theorem.

\begin{theorem}\label{detQ}
Let $G$ be a simple graph with signless Laplacian matrix $Q$. Then 
$$\det(Q)=\sum_{H} 4^{c(H)},$$
where the summation runs over all spanning subgraphs $H$ of $G$ whose all connected components are odd-unicyclic.
\end{theorem}

\begin{proof}
By Theorem \ref{CB} and Observation \ref{CB3},
$$\det(Q)=\det(NN^T)=\sum_{S} \det(N(;S])^2,$$
where the summation runs over all $n$-subsets $S$ of $\{1,2,\ldots,m\}$. By Theorem \ref{oddunicycle}, we have
$$\det(Q)=\sum_{S} \det(N(;S])^2=\sum_{H} (\pm 2^{c(H)})^2=\sum_{H} 4^{c(H)},$$
where the summation runs over all spanning subgraphs $H$ of $G$ whose all connected components are odd-unicyclic. 
\end{proof}

Let $ous(G)$ denote the number of spanning subgraphs $H$ of a graph $G$ where each connected component of $H$ is an odd-unicyclic graph. So $ous(G)$ is the number of $TU$-subgraphs of $G$ whose all connected components are odd-unicyclic. Note that $c(H)\geq 1$ for all spanning subgraphs $H$ of $G$ whose all connected components are odd-unicyclic. By Theorem \ref{detQ}, we have an upper bound for $ous(G)$.

\begin{corollary}
Let $G$ be a simple graph with signless Laplacian matrix $Q$. Then $\det(Q)\geq 4ous(G)$.
\end{corollary}

For example, if $G$ is bipartite graph, then $\frac{\det(Q)}{4}=0=ous(G)$. If $G$ is an odd-unicyclic graph, then $\frac{\det(Q)}{4}=1=ous(G)$.\\

Note that by appending edges to an odd cycle in $G$ we get at least one $TU$-subgraph of $G$ with a unique odd-unicyclic connected component.   Let $oc(G)$ denote the number of odd cycles in a graph $G$. Then $oc(G)\leq ous(G)$, where the equality holds if and only if $G$ is a bipartite graph or an odd-unicyclic graph. Then we have the following corollary.

\begin{corollary}\label{numberofoddcycles}
Let $G$ be a simple graph with signless Laplacian matrix $Q$. Then  $\frac{\det(Q)}{4}\geq oc(G)$, the number of odd cycles in $G$, where the equality holds if and only if $G$ is a bipartite graph or an odd-unicyclic graph.
\end{corollary}

\section{Open Problems}
In this section we pose some problems related to results in Sections 2 and 3. First recall Corollary \ref{numberofoddcycles} which gives a linear algebraic sharp upper bound for the number of odd cycles in a graph. So an immediate question would be the following:
	
\begin{question}
Find a linear algebraic (sharp) upper bound of the number of even cycles in a simple graph.
\end{question}

To answer this one may like to apply Cauchy-Binet Theorem as done in Sections 2 and 3. Then a special $n\times m$ matrix $R$ will be required with the following properties:
\begin{enumerate}
\item $RR^T$ is a decomposition of a fixed matrix for a given graph $G$. 
\item If $G$ is an even (resp. odd) cycle, then $\det(R)$ is $\pm c$ (resp. zero) for some fixed nonzero number $c$.
\end{enumerate}

For other open questions consider a simple connected graph $G$ on $n$ vertices and $m\geq n$ edges with signless Laplacian matrix $Q$. The characteristic polynomial of $Q$ is 
$$P_Q(x)=\det(x I_n-Q)=x^n+\sum_{i=1}^n a_i x^{n-i}.$$
It is not hard to see that $a_1=-2m$ and $a_2=2m^2-m-\frac{1}{2}\sum_{i=1}^nd_i^2$ where $(d_1,d_2,\ldots,d_n)$ is the degree-sequence of $G$. 
Theorem 4.4 in \cite{CRC} provides a broad combinatorial interpretation for $a_i$, $i=1,2,\ldots,n$. A combinatorial expression for $a_3$ is obtained in \cite[Thm 2.6]{WHAB} by using mainly Theorem 4.4 in \cite{CRC}. Note that
$$a_3=(-1)^3\disp\sum_{1\leq i_1<i_2<i_3\leq n} \det(Q[i_1,i_2,i_3]).$$

So it may not be difficult to find corresponding combinatorial interpretation of $\det(Q[i_1,i_2,i_3])$ in terms of subgraphs on three edges. Similarly we can investigate other coefficients and corresponding minors which  we essentially did for $a_{n}$ and $a_{n-1}$ in Sections 3 and 2 respectively. So the next coefficient to study is $a_{n-2}$ which entails the following question:

\begin{question}
Find a combinatorial expression or a lower bound for $\det(Q(i_1,i_2))$.
\end{question}

By Cauchy-Binet Theorem,
 $$\det(Q(i_1,i_2))=\sum_{S} \det(N(i_1,i_2;S])^2,$$
where the summation runs over all $(n-2)$-subsets $S$ of the edge set $\{1,2,\ldots,m\}$. So it comes down to finding a combinatorial interpretation of  $\det(N(i_1,i_2;S])$.

\end{document}